\documentclass[12pt]{article}
\usepackage{amsfonts}
\usepackage{times}
\usepackage[left=1in,top=1in,right=1in]{geometry}
\usepackage{amsmath,amsthm,amssymb}
\usepackage{color}
\usepackage{latexsym}
\usepackage{cite}
\theoremstyle{plain}

\newtheorem{theorem}{Theorem}
\theoremstyle{plain}

\newtheorem{corollary}{Corollary}

\newtheorem{definition}{Definition}
\newtheorem{example}{Example}

\newtheorem{proposition}{Proposition}

\newcommand{\NKCMM}{$ N(\kappa) $-contact metric manifold}

\date{}
\newcommand{\NK}{$ N(\kappa)$- contact metric manifolds}
\newcommand{\T}{$\mathcal{T}$-curvature tensor}
\begin{document}
\title{A classification of $ N(\kappa)$-contact metric manifolds with $ \mathcal{T} $-curvature tensor }
\author{\textbf{ \.Inan \"Unal} \\     
	{\normalsize Department of computer Engineering}\\{\normalsize  Munzur University, Turkey} \\{\normalsize inanunal@munzur.edu.tr}\\ \\ \textbf{Mustafa Alt{\i}n}\\{\normalsize Technical Sciences Vocational School}\\{\normalsize Bing\"ol University, Turkey}\\ {\normalsize maltin@bingol.edu.tr}\\ \\\textbf{
		Shashikant Pandey}\\{\normalsize  Department of Mathematics and Astronomy }\\{\normalsize University of Lucknow, India}\\ {\normalsize shashi.royal.lko@gmail.com}\\\ }
\maketitle


\noindent \textbf{Abstract:} In this paper, we present a classification of \NKCMM s with using some special flatness conditions on \T. \ We examine $\mathcal{T}$-flat, quasi-$\mathcal{T}$-flat, $ \xi$-$\mathcal{T}$-flat and $ \varphi $-$\mathcal{T}$-flat \NKCMM s.Also,we consider the conditions $  \mathcal{T}(\xi,X).R=0 $ and $ \mathcal{T}(\xi,X).S=0 $ for the Riemannian curvature tensor $ R $ and Ricci curvature tensor $ S $. Thus, we obtain a classification of \NKCMM s. 


\section{Introduction}

A tensor $ \mathcal{T} $ of a type $ (1,3) $ on a Riemannian manifold $ M $ is called a curvature tensor if it satisfies following properties;	
\begin{enumerate}
	\item [$ \bullet $]$ \mathcal{T}( X_1,X_2) =-\mathcal{T}( X_2,X_1)  $
	\item [$ \bullet $]$ 	g( \mathcal{T}( X_1,X_2) X_3,X_4) =-g( \mathcal{T}( X_1,X_2)
	X_4,X_3) $
	\item [$ \bullet $] $ \mathcal{T}( X_1,X_2) X_3+\mathcal{T}( X_2,X_3) X_1+\mathcal{T}( X_3,X_1) X_2=0 $
\end{enumerate}
for all $ X_1,X_2,X_3,X_4 \in \Gamma(TM) $. The Riemannian curvature tensor is a well known example. Except for Riemannian curvature tensor, there are different kinds of curvature tensors. We examine the Riemannian geometry of the manifolds by using some algebraic and differential properties of these curvature tensors such as Weyl conformal curvature tensor , conharmonic curvature tensor concircular curvature tensors etc. In 2011 Tripathi and Gupta \cite{MMTripathigupta} defined a new curvature tensor which is called by $ \mathcal{T} $-curvature tensor as follows;
\begin{align}
\mathcal{T}(X_{1},X_{2})X_{3}  &  =a_{0}R(X_{1},X_{2})X_{3}+a_{1}S(X_{2}%
,X_{3})X_{1}+a_{2}S(X_{1},X_{3})X_{2}\label{gT-tensor}\\
&  +a_{3}S(X_{1},X_{2})X_{3}+a_{4}g(X_{2},X_{3})QX_{1}+a_{5}g(X_{1}%
,X_{3})QX_{2}\nonumber\\
&  +a_{6}g(X_{1},X_{2})QX_{3}+a_{7}r\left(  g(X_{2},X_{3})X_{1}-g(X_{1}%
,X_{3})X_{2}\right) \nonumber
\end{align}
where, $a_{0}$, ..., $a_{7}$ are some smooth functions on $M$, $R,S,Q$ and
$r$ are the Riemannian curvature tensor, the Ricci tensor, the Ricci operator of type $(1,1)$ and the scalar curvature, respectively. We obtain different type of curvature tensors for the special values of $a_{0}$, ..., $a_{7}$ which we present the next section. Thus, 
$ \mathcal{T} $-curvature tensor is a generalization of special curvature tensors. 
\par 
Tripathi and Gupta \cite{MMTripathigupta}  classify the $ K $-contact and Sasakian manifolds under the some certain conditions on $ \mathcal{T} $-curvature tensor. Also same authors worked on \T \ on semi-Riemann manifolds \citen{tripathi2011t}. In 2012 Nagaraja and Somashekhara \cite{HGNagara} obtained some results on $ (\kappa, \mu) $-contact manifolds with \T.  G. Ingalahalli andd Bagewadi studied on $ \varphi $-symmetric $ \mathcal{T} $- curvature tensor $ N(\kappa)- $contact metric manifolds. Ravikumar et al. classified Lorentzian $ \alpha$-Sasakian manifolds with $ \mathcal{T}$-curvature tensor. In 2019  Gupta studied on $ \varphi $-$ \mathcal{T}$-Symmetric $ \epsilon$-Para Sasakian manifolds. As can be seen from these studies, with \T \ we can classify manifolds with special structures. A \NKCMM\ is an almost contact manifold with the distribution;  
\begin{equation*}
N_{p}(\kappa)=\{X_{3}\in \Gamma
(T_{p}M):R(X_{1},X_{2})X_{3}=\kappa[g(X_{2},X_{3})X_{1}-g(X_{1},X_{3})X_{2}]\}
\end{equation*}
for all $ X_1,X_2 \in \Gamma(TM) $  \cite{blair2005concircular}. In a \NKCMM \ $ \xi $ belongs to $ N_{p}(\kappa) $. A \NKCMM \ is Sasakian if $ \kappa=1$. The Riemannian geometry of \NKCMM\ have been studied by many researchers \cite{de2010weyl,de2014bochner,de2014class,de2018certain,barman2017n,ingalahalli2019certain,unal2020n}. \par 
Motivated by all these work, we present a classification of \NKCMM s with using some special flatness conditions on \T. \ We consider $\mathcal{T}$-flat, quasi-$ \mathcal{T} $-flat, $ \xi- $$ \mathcal{T} $-flat and $ \varphi- $$ \mathcal{T} $-flat \NKCMM s. Also, we examine the conditions $ \mathcal{T}(\xi,X).R=0 $ and $ \mathcal{T}(\xi,X).S=0 $ for the Riemannian curvature tensor $ R $ and Ricci curvature tensor $ S $.  Thus, we obtain a classification of \NKCMM s with special values of  $a_{0}$, ..., $a_{7}$.

\section{Preliminaries}
Let $M$ be a $(2n+1)-$dimensional smooth manifold. For a $(1,1)$ tensor field $\phi $,  a vector field $ 
\xi $ and a $1-$ form $\eta $  on $M$, $(\phi ,\xi ,\eta )$ is
called an almost contact structure on $ M $ if we have 
\begin{equation*}
\phi ^{2}X=X-\eta (X)\xi \ \ ,\ \phi (\xi )=0,\ \ \ \ \eta \circ \phi =0\ \
\ ,\ \eta (\xi )=1.
\end{equation*}

The kernel of $\eta $ defines a non-integrable distribution on $M$ and the distribution is called by contact distribution. The rank of $\phi $ is $2m$. The Riemannian metric $ g $ is called associated metric if $ g(\phi X_{1},\phi
X_{2})=-g(X_{1},X_{2})+\eta (X_{1})\eta (X_{2})$ and it is called compatible  metric if $ d\eta  (X_{1},X_{2})=g(\phi X_{1},X_{2}) $ for all $X_{1},X_{2}\in \Gamma (TM)$. The
manifold $M$ is called by almost contact metric manifold with the structure $%
(\phi ,\xi ,\eta )$ and associated metric $g$. \par 
If characteristic vector field $ \xi $ is Killing vector field then $ M $ is called $ K- $contact. That is in a $ K- $contact manifold $ h=0 $.  An almost contact metric manifold $M$ is said to be normal if $ \phi $ is integrable. Also when the contact manifold is normal $ h=0 $ .
If the second fundamental form $ \Omega $ of an almost contact metric manifold $ M $ is $ \Omega(X_1,X_2 )=g(\phi X_1,X_2)$ and $ M $ is normal then $ M $ is called Sasakian. A Sasakian manifold is a $ K- $contact manifold, but the converse holds only if $ dimM^{(2n+1)}=3 $.\par 
The curvature relations on a \NKCMM \ are given by  
\begin{eqnarray*}
R(X_{1},\xi)\xi &=&\kappa [X_{1}-\eta(X_1)\xi] \label{R(X,xi)xi}\\
R(X_{1},X_{2})\xi &=&\kappa\left[ \eta (X_{2})X_{1}-\eta (X_{1})X_{2}\right],\label{R(X,Y)xi}\\
R(X_{1},\xi )X_{2}&=&-\kappa\left[ g(X_{1},X_{2})\xi -\eta (X_{2})X_{1}\right]\label{R(X,xi)Y}. 
\end{eqnarray*}
where $ R $ is the Riemannian curvature tensor of $ M $ which is defined by 
\begin{equation*}
R(X_{1},X_{2})X_3=\nabla_{X_1}\nabla_{X_2}X_3-\nabla_{X_2}\nabla_{X_1}X_3-\nabla_{[ X_1,X_2]}X_3
\end{equation*}
for all $ X_{1},X_{2}, X_3 \in \Gamma(TM)$. Also the Ricci curvature  $ S $ and scalar curvature of $ M $ is given by; 
\begin{eqnarray}
S(X_{1},X_{2})&=&2(n-1)g(X_{1},X_{2})+2(n-1)g(hX_{1},X_{2})\label{RicciX,Y}\\&&+\left[
2n\kappa-2(n-1)\right] \eta (X_{1})\eta (X_{2})\notag\\
S(X_1,\xi)&=&2\kappa n\eta(X_1), S(\xi,\xi)=2\kappa n\label{RicX,xiandRicxi,xi}  \\
\tau&=&2n(2n-2+\kappa)\notag 
\end{eqnarray}
for all $ X_{1},X_{2} \in \Gamma(TM)$ \cite{blair2005concircular}. We will use the following basic equalities from Riemann geometry; 
\begin{eqnarray}
{\displaystyle\sum\limits_{i=1}^{2n}}
g(e_{i},e_{i})&=&
{\displaystyle\sum\limits_{i=1}^{2n}}
g(\varphi e_{i},\varphi e_{i})=2n, \label{Mgei}\\{\displaystyle\sum\limits_{i=1}^{2n}}
g(e_{i},X_{3})S(X_{2},e_{i})&=&
{\displaystyle\sum\limits_{i=1}^{2n}}
g(\varphi e_{i},X_{3})S(X_{2},\varphi e_{i})=S(X_{2},X_{3})-S(X_{2},\xi
)\eta(X_{3}), \label{Mgei2}\\
{\displaystyle\sum\limits_{i=1}^{2n}}
g(e_{i},X_{3})S(X_{2},e_{i})&=&
{\displaystyle\sum\limits_{i=1}^{2n}}
g(\varphi e_{i},\varphi X_{3})S(X_{2},\varphi e_{i})=S(X_{2},\varphi X_{3}),
\label{Sxe}.
\end{eqnarray}
\begin{definition}
An almost contact metric manifold $M$ is said to be $\eta-$Einstein manifold if the Ricci tensor satisfies%
\begin{equation*}
S(X_{1},X_{2})=b_{1}g(X_{1},X_{2})+b_{2}\eta(X_{1})\eta(X_{2}), \label{Sxy}
\end{equation*}
where $b_{1,}b_{2}$ are smooth funcstion on the manifold. 
\end{definition}
In part\i cular, if $b_{2}=0,$ then $M$ becomes an Einstein manifold. Thus, $ \eta- $Einstein manifolds are natural generalization of Einstein manifolds.\par  
In \cite{blair1995contact}, Blair et al. showed that $ (\kappa, \mu)- $nullity distribution is invariant under 
\begin{equation}\label{barkappavebarmu}
\bar{\kappa}=\frac{\kappa+a^2-1}{a}, \ \  \bar{\mu}=\frac{\mu+2c-2}{a}.
\end{equation}
In \cite{EBoeck} Boeckx introduced the number $ I_{M}=\frac{1-\frac{\mu}{2}}{\sqrt{1-k}} $ for non-Sasakian $(k,\mu)$-contact manifold. This number is called by Boeckx invariant. There are two classes in the classification of non-Sasakian $ (\kappa,\mu)- $spaces. The first class is a manifold with constant sectional curvature $ c $. In this case $ \kappa=c(2-c) $ and $ \mu=-2c $ and by this we get an example of \NKCMM. \  The second class is on $ 3- $dimensional  Lie groups. Boeckx proved that  two Boeckx invariant of two non-Sasakian $ (\kappa,\mu)- $space are equal if and only if this manifolds  are locally isometric as contact metric manifolds.  
Blair, Kim and Tripathi \cite{blair2005concircular} gave following example of \NKCMM s\  by using the Boeckx invariant for the first class. 
\begin{example}\label{example1}
	The Boeckx invariant for a $N(1-\frac{1}{n},0)$-manifold is $\sqrt{n}%
	>-1$.  By consider the tangent sphere bundle of an $(n+1)$-dimensional
	manifold of constant curvature $c$, as the resulting $D$%
	-homothetic deformation is $\kappa=c(2-c)$, $ \mu=-2c $ and from ( \ref{barkappavebarmu} ) we get 
	\begin{equation*}
	c=\frac{(\sqrt{n}\pm 1)^2}{n-1},\text{ \ \ \ \ \ \ }a=1+c.
	\textsl{}	\end{equation*}%
	and taking $c$ and $a$ to be these values we obtain $N(1-\frac{1}{n})$%
	-contact metric manifold.
\end{example}
For the \T\  is defined as in (\ref{gT-tensor}), we have 
\begin{align}
\mathcal{T}(X_{1},X_{2},X_{3},X_{4}) &  =a_{0}R(X_{1},X_{2},X_{3},X_{4})+a_{1}%
S(X_{2},X_{3})g(X_{1},X_{4})\label{Txyzv}\\
&  +a_{2}S(X_{1},X_{3})g(X_{2},X_{4})+a_{3}S(X_{1},X_{2})g(X_{3}%
,X_{4})\nonumber\\
&  +a_{4}g(X_{2},X_{3})S(X_{1},X_{4})+a_{5}g(X_{1},X_{3})S(X_{2}%
,X_{4})\nonumber\\
&  +a_{6}S(X_{3},X_{4})g(X_{1},X_{2})\nonumber\\
&  +a_{7}r(g(X_{2},X_{3})g(X_{1},X_{4})-g(X_{1},X_{3})g(X_{2},X_{4}%
)).\nonumber
\end{align}
where $ \mathcal{T}(X_{1},X_{2},X_{3},X_{4})=g(\mathcal{T}(X_{1},X_{2})X_{3},X_{4}) $ and with the special values of coefficients $ a_i, \  1\leq i\leq 7$  the $ \mathcal{T} $-curvature tensor is reduced to quasi-conformal, conformal, conharmonic, concircular, pseudo-projective,
projective, $M$-projective, $W_{i}$-curvature tensors $(i=0,...,9)$ $W_{j}%
$-curvature tensors $(j=0,1)$quasi-conformal, conformal, conharmonic, concircular, pseudo-projective,
projective, $M$-projective, $W_{i}$-curvature tensors $(i=0,...,9)$ $W_{j}%
$-curvature tensors $(j=0,1)$ as follow:
\begin{center}
\begin{table}
	\footnotesize 	\begin{tabular}{|c|c|}
	
	\hline \textbf{\T}
	&  \textbf{$ a_i $} coefficients\\
		\hline  \textit{ quasi-conformal curvature tensor }  $C_{\ast}$ \cite{yano1968riemannian}
		& $ 	a_{1}=-a_{2}=a_{4}=-a_{5,}\ a_{3}=a_{6}=0,\  a_{7}=-\frac{1}%
		{2n+1}(\frac{a_{0}}{2n}+2a_{1}) $ \\
		\hline \textit{ conformal curvature tensor }$C$ \cite{YIshii}
		& $ a_{0}=1,\ a_{1}=-a_{2}=a_{4}=-a_{5}=-\frac{1}{2n-1},\  a_{3}=a_{6}=0, a_{7}=\frac{1}{2n(2n-1)} $ \\
		\hline \textit{conharmonic curvature tensor} $L$ \cite{YIshii}
		&  $ a_{0}=1,\ a_{1}=-a_{2}=a_{4}=-a_{5}=-\frac{1}{2n-1}, \ a_{3}=a_{6}=0,\text{ \ \ \ \ }a_{7}=0  $\\
		\hline\textit{concircular curvature tensor} $V$ \cite{yano1968riemannian}
		& $ 	a_{0}=1,\ a_{1}=a_{2}=a_{3}=a_{4}=a_{5}=a_{6}=0,\ a_{7}=-\frac{1}{2n(2n+1)} $ \\
		\hline \textit{pseudo-projective curvature tensor} $P_{\ast}$ \cite{prasad2002pseudo}
		& $ 
			 a_{1}=-a_{2},\ a_{3}=a_{4}=a_{5}=a_{6}=0,	a_{7}=-\frac{1}{2n+1}(\frac{a_{0}}{2n}+a_{1}) $\\
		\hline \textit{projective curvature tensor} $P$ \cite{yano2016curvature}
		&  $ a_{0}=1,\ a_{1}=-a_{2}=-\frac{1}{2n}, \ a_{3}=a_{4}=a_{5}=a_{6}=a_{7}=0 $\\
		\hline \textit{$M-$projective curvature tensor} \cite{pokhariyal1971curvature}
		&  $ a_{0}=1,\ a_{1}=-a_{2}=a_{4}=-a_{5}=-\frac{1}{4n},\ a_{3}=a_{6}=a_{7}=0 $\\
		\hline \textit{$W_{0}-$projective curvature tensor} \cite{pokhariyal1971curvature}
		&  $ a_{0}=1,\  a_{1}=-a_{5}=-\frac{1}{2n},\ a_{2}=a_{3} 	=a_{4}=a_{6}=a_{7}=0 $\\
		\hline \textit{ $W_{0}^{\star}-$projective curvature tensor} \cite{pokhariyal1971curvature} 
		& $ a_{0}=1,\  a_{1}=-a_{5}=-\frac{1}{2n},\  a_{2}=a_{3}	=a_{4}=a_{6}=a_{7}=0 $ \\
		\hline  \textit{$W_{1}-$projective curvature tensor} \cite{pokhariyal1971curvature}
		&  $ a_{0}=1,\ a_{1}=-a_{2}=\frac{1}{2n},\ 
		a_{3}=a_{4}=a_{5}=a_{6}=a_{7}=0 $\\
		\hline \textit{$W_{1}^{\star}-$projective curvature tensor}		\cite{pokhariyal1971curvature}
		& $ a_{0}=1,\ a_{1}=-a_{2}=\frac{-1}{2n},\ 	a_{3}=a_{4}=a_{5}=a_{6}=a_{7}=0 $\\
		\hline  \textit{$W_{2}-$projective curvature tenso}r \cite{pokhariyal1970curvature} 
		& $ 	a_{0}=1,\ a_{4}=-a_{5}=\frac{-1}{2n},\ 	a_{1}=a_{2}=a_{3}=a_{6}=a_{7}=0 $ \\
		\hline \textit{$W_{3}-$projective curvature tensor} \cite{pokhariyal1971curvature}
		&  $ 	a_{0}=1,\text{ \ \ \ \ \ \ }a_{2}=-a_{4}=\frac{-1}{2n},\text{ \ \ \ \ \ \ }%
			a_{1}=a_{3}=a_{5}=a_{6}=a_{7}=0 $\\
		\hline \textit{$W_{4}-$projective curvature tensor} \cite{pokhariyal1971curvature}
		& $ a_{0}=1,\text{ \ \ \ \ \ \ }a_{5}=-a_{6}=\frac{1}{2n},\text{ \ \ \ \ \ \ }%
		a_{1}=a_{2}=a_{3}=a_{4}=a_{7}= $  \\
		\hline \textit{$W_{5}-$projective curvature tensor} \cite{pokhariyal1982relativistic}
		& $ a_{0}=1,\text{ \ \ \ \ \ \ }a_{2}=-a_{5}=\frac{-1}{2n},\text{ \ \ \ \ \ \ }%
		a_{1}=a_{3}=a_{4}=a_{6}=a_{7}=0 $\\
		\hline \textit{ $W_{6}-$projective curvature tensor} \cite{pokhariyal1982relativistic}
		& $ a_{0}=1,\text{ \ \ \ \ \ \ }a_{1}=-a_{6}=\frac{-1}{2n},\text{ \ \ \ \ \ \ }%
		 a_{2}=a_{3}=a_{4}=a_{5}=a_{7}=0 $\\
		\hline \textit{$W_{7}-$projective curvature tensor} \cite{pokhariyal1982relativistic}
		& $ a_{0}=1,\text{ \ \ \ \ \ \ }a_{1}=-a_{4}=\frac{-1}{2n},\text{ \ \ \ \ \ \ }%
		a_{2}=a_{3}=a_{5}=a_{6}=a_{7}=0 $ \\
		\hline \textit{$W_{8}-$projective curvature tensor} \cite{pokhariyal1982relativistic}
		&  $ a_{0}=1,\text{ \ \ \ \ \ \ }a_{1}=-a_{3}=\frac{-1}{2n},\text{ \ \ \ \ \ \ }%
		a_{2}=a_{4}=a_{5}=a_{6}=a_{7}=0 $\\
		\hline \textit{$W_{9}-$projective curvature tensor} \cite{pokhariyal1982relativistic}
		&  $ a_{0}=1,\text{ \ \ \ \ \ \ }a_{3}=-a_{4}=\frac{-1}{2n},\text{ \ \ \ \ \ \ }%
		a_{1}=a_{2}=a_{5}=a_{6}=a_{7}=0 $\\
		\hline
	\end{tabular}
\caption{\label{Det.T} Determination of \T \ with special values of $ a_i $}
\end{table}
\end{center}

\section{$ \mathcal{T} $-flatness on $ N(\kappa) $-contact metric manifolds }
A flat manifold has vanishing Riemannian curvature tensor and also it is locally isometric to Euclidean space. Similarly, the vanishing of curvature tensors like conformal, concircular, conharmonic, etc. has many geometric interpretations. Also, the flatness of curvature tensors could be generalized by the flatness of $ \mathcal{T} $-tensor. In this section we consider $ N(\kappa) $- contact metric manifolds under the many flatness conditions of $ \mathcal{T} $-tensor. 

\begin{definition}
A $N(k)$ contact metric manifold M is said to be
\begin{enumerate}
	\item $ \mathcal{T} $-flat if $ \mathcal{T}(X_{1},X_{2})X_3=0$,
	\item $\xi-\mathcal{T}$-flat if $ 	\mathcal{T}(X_{1},X_{2})\xi=0, $
	\item quasi $\mathcal{T}$-flat if	$ g(\mathcal{T}(\varphi X_{1},X_{2})X_{3},\varphi X_{4})=0$
	\item $\varphi-T$-flat if	$g(\mathcal{T}(\varphi X_{1},\varphi X_{2})\varphi X_{3},\varphi X_{4})=0$
\end{enumerate}
for every $ X_1,X_2,X_3,X_4 \in \Gamma(TM). $
\end{definition}

\begin{proposition}
	Let $ M $ a $ \mathcal{T} $-flat \NKCMM. \ Then we have 
	\begin{equation*}
		a_{1}(2n-2+ \kappa)+(a_{0}+a_{2}+a_{3}+(2n+1)a_{4}+a_{5}+a_{6}) \kappa +(2n(2n-2+ \kappa ))a_{7}=0.
	\end{equation*}
\end{proposition}
\begin{proof}
	Suppose that $ M $ is a $ \mathcal{T} $-flat \NKCMM.\ Then we have 
	\begin{align*}
	0 &  =a_{0}S(X_{1},X_{4})+a_{1}rg(X_{1},X_{4})+a_{2}S(X_{1},X_{4}%
	)+a_{3}S(X_{1},X_{4})\\
	&  +a_{4}(2n+1)S(X_{1},X_{4})+a_{5}S(X_{1},X_{4})+a_{6}S(X_{1},X_{4}%
	)+2na_{7}rg(X_{1},X_{4}).
	\end{align*}
	By putting $X_{1}=X_{4}=\xi$ in last equation, we get 
\begin{equation*}
	a_{1}(2n-2+ \kappa)+(a_{0}+a_{2}+a_{3}+(2n+1)a_{4}+a_{5}+a_{6}) \kappa+(2n(2n-2+ \kappa))a_{7}=0. 
\end{equation*}
\end{proof}
Thus, we get following; 
\begin{theorem}
	Let $ M $ be $ \mathcal{T} $-flat \NKCMM. \ Then, we have the classification in Table \ref{T-flat}. 
\begin{center}
\begin{table}
	\begin{center}
			\begin{tabular}{|l|l|l|c|c|c|}\hline
	\textbf{\NKCMM} & \textbf{$\kappa $} \\ \hline 
	$g(\mathcal{C}_{\star }(X_{1},X_{2})X_{3},X_{4})=0$ & $\frac{n-1}{n}$ \\ \hline 
	$g(\mathcal{C}(X_{1},X_{2})X_{3},X_{4})=0$ & a real number \\ \hline 
	$g(\mathcal{L}(X_{1},X_{2})X_{3},X_{4})=0$ & $2-2n$ \\ \hline 
	$g(\mathcal{V}(X_{1},X_{2})X_{3},X_{4})=0$ & $\frac{n-1}{n}$ \\ \hline 
	$g(\mathcal{P}_{\star }(X_{1},X_{2})X_{3},X_{4})=0$ & $\frac{n-1}{n}$ \\ \hline 
	$g(\mathcal{P}(X_{1},\varphi X_{2})X_{3},X_{4})=0$ & $\frac{n-1}{n}$ \\ \hline 
	$g(\mathcal{M}(X_{1},X_{2})X_{3},X_{4})=0$ & $\frac{n-1}{n}$ \\ \hline 
	$g(\mathcal{W}_{0}(X_{1},X_{2})X_{3},X_{4})=0$ & $\frac{n-1}{n}$ \\ \hline 
	$g(\mathcal{W}_{0}^{_{\star }}(X_{1},X_{2})X_{3},X_{4})=0$ & $\frac{1-n}{n}$
	\\ \hline 
	$g(\mathcal{W}_{1}(X_{1},X_{2})X_{3},X_{4})=0$ & $\frac{1-n}{n}$ \\ \hline 
	$g(\mathcal{W}_{1}^{_{\star }}(X_{1},X_{2})X_{3},X_{4})=0$ & $\frac{n-1}{n}$
	\\ \hline 
	$g(\mathcal{W}_{2}(X_{1},X_{2})X_{3},X_{4})=0$ & a real number \\ \hline 
	$g(\mathcal{W}_{3}(X_{1},X_{2})X_{3},X_{4})=0$ & 0 \\ \hline 
	$g(\mathcal{W}_{4}(X_{1},X_{2})X_{3},X_{4})=0$ & 0 \\ \hline 
	$g(\mathcal{W}_{5}(X_{1},X_{2})X_{3},X_{4})=0$ & 0 \\ \hline 
	$g(\mathcal{W}_{6}(X_{1},X_{2})X_{3},X_{4})=0$ & $\frac{n-1}{n}$ \\ \hline 
	$g(\mathcal{W}_{7}(X_{1},X_{2})X_{3},X_{4})=0$ & $\frac{n-1}{2n}$ \\ \hline 
	$g(\mathcal{W}_{8}(X_{1},X_{2})X_{3},X_{4})=0$ & $\frac{n-1}{n}$ \\ \hline 
	$g(\mathcal{W}_{9}(X_{1},X_{2})X_{3},X_{4})=0$ &  a real number\\ \hline 
\end{tabular} 
	\end{center}
\caption{\label{T-flat} Classification of $ \mathcal{T} $-flat \NKCMM\ with special values of $ a_i $}
\end{table}

\end{center}
\end{theorem}

\begin{corollary}
	Let $ M $ be a \NKCMM. \ 
	\begin{enumerate}
		\item If $ M $ is  $ \mathcal{C}_*- $flat,  $ \mathcal{V}- $flat, $ \mathcal{P}_*- $flat, $ \mathcal{P}- $flat, $ \mathcal{M}- $flat, $ \mathcal{W}_{0}- $flat, $\mathcal{W}_{1}^{_{\star }}-$flat, $ \mathcal{W}_{6}- $flat and  $ \mathcal{W}_{8}- $flat then it is locally isometric to Example 1. 
		\item If $ M $ is $ \mathcal{W}_3- $flat, $ \mathcal{W}_4- $flat  or $ \mathcal{W}_5- $flat then it is locally isometric to $ E^{(n+1)} \times S^n(4) $. 
	\end{enumerate}
\end{corollary}
\begin{theorem}
A \NKCMM \ is quasi$-T$-flat if  $ C_1=a_{0}+2na_{1}+a_{2}+a_{3}+a_{5}+a_{6}\neq0 $. 
Then we have 
\[
S=\frac{A_1}{C_1}g+\frac{B_1}{C_1}\eta\otimes\eta,
\]
where%
\[
A_1=a_{0}\kappa+a_{4}(2n\kappa-r)+a_{7}r(1-2n)
\]
and%
\[
B_1=-a_{0}\kappa+2n\kappa(a_{2}+a_{3}+a_{5}+a_{6})-a_{7}r
\]
Therefore $M~$is an $\eta-$Einstein manifold. Consequently, we have the classification in Table \ref{quasiT-flat}. 
\end{theorem}
\begin{table}[h]\small
\begin{tabular}{|l|l|l|c|c|c|}\hline
$N(\kappa)-$contact manifold & $\!\begin{aligned}[t] \eta-\text{Einstein}/ \\ \text{Einstein} \end{aligned}$ & $S=$\\\hline
$g(\mathcal{C}_{\star}(\varphi X_{1},X_{2})X_{3},\varphi X_{4})=0$ &$\eta-$Einstein  & \footnotesize $\!\begin{aligned}[t]
&(\frac{\frac{(2 n-1) (\kappa+2 n-2) (4 a_{1} n+a_{0})}{2 n+1}+2 \kappa a_{1} n-2 a_{1} n (\kappa+2 n-2)+\kappa}{2 a_{1} (n-1)+1})g\\&+(\frac{\kappa \left(-8 a_{1} n^2-2 n+a_{0}-1\right)+2 (n-1) (4 a_{1} n+a_{0})}{(2 n+1) (2 a_{1} (n-1)+1)})\eta\otimes\eta
\end{aligned}$
\\\hline
$g(\mathcal{C}(\varphi X_{1},X_{2})X_{3},\varphi X_{4})=0$ & $\eta-$Einstein &
$(2n-2)g+2(n\kappa-n+1)\eta\otimes\eta$\\\hline
$g(\mathcal{L}(\varphi X_{1},X_{2})X_{3},\varphi X_{4})=0$ & $\eta-$Einstein &
$(\kappa (2 n-1)+4 (n-1) n)g+(2n\kappa+\kappa)\eta\otimes\eta$\\\hline
$g(\mathcal{V}(\varphi X_{1},X_{2})X_{3},\varphi X_{4})=0$ & $\eta-$Einstein &
$(\frac{2n(2n-3+2\kappa)+2}{2n+1})g+(\frac{2n-2n\kappa-2}{2n+1})\eta\otimes\eta$\\\hline
$g(\mathcal{P}_{\star}(\varphi X_{1},X_{2})X_{3},\varphi X_{4})=0$ & $\eta
-$Einstein & $\!\begin{aligned}[t]
&\footnotesize (\frac{\frac{(2 n-1) (\kappa+2 n-2) (2 a_{1} n+a_{0})}{2 n+1}+\kappa a_{0}}{a_{1} (2 n-1)+a_{0}})g+\\&(-\frac{2 ((\kappa-1) n+1) (2 a_{1} n+a_{0})}{(2 n+1) (a_{1} (2 n-1)+a_{0})})\eta\otimes\eta
\end{aligned}$ \\\hline
$g(\mathcal{P}(\varphi X_{1},X_{2})X_{3},\varphi X_{4})=0$ & Einstein &
$(2n\kappa)g$\\\hline
$g(\mathcal{M}(\varphi X_{1},X_{2})X_{3},\varphi X_{4})=0$ & Einstein &
$(\frac{2n(\kappa+n-1)}{n+1})g$\\\hline
$g(\mathcal{W}_{0}(\varphi X_{1},X_{2})X_{3},\varphi X_{4})=0$ & Einstein  &
$(2n\kappa)g$\\\hline
$g(\mathcal{W}_{0}^{_{\star}}(\varphi X_{1},X_{2})X_{3},\varphi X_{4})=0$ &
$\eta-$Einstein &$(\frac{2n\kappa}{4n-1})g+(\frac{4n\kappa}{1-4n})\eta\otimes\eta$\\\hline
$g(\mathcal{W}_{1}(\varphi X_{1},X_{2})X_{3},\varphi X_{4})=0$ &  $\eta-$ Einstein &
$(\frac{2n\kappa}{4n-1})g+(\frac{4n\kappa}{1-4n})\eta\otimes\eta$\\\hline
$g(\mathcal{W}_{1}^{_{\star}}(\varphi X_{1},X_{2})X_{3},\varphi X_{4})=0$ &
Einstein & $(2n\kappa)g$\\\hline
$g(\mathcal{W}_{2}(\varphi X_{1},X_{2})X_{3},\varphi X_{4})=0$ & Einstein &
$(\frac{2n(2n-2+\kappa)}{2n+1})g$\\\hline
$g(\mathcal{W}_{3}(\varphi X_{1},X_{2})X_{3},\varphi X_{4})=0$ & $\eta
-$Einstein & $(\frac{2n(2-2n+\kappa)}{2n-1})g+(\frac{-4n\kappa}{2n-1})\eta\otimes\eta
$\\\hline
$g(\mathcal{W}_{4}(\varphi X_{1},X_{2})X_{3},\varphi X_{4})=0$ & $\eta
-$Einstein & $(\kappa)g+(-\kappa)\eta\otimes\eta$\\\hline
$g(\mathcal{W}_{5}(\varphi X_{1},X_{2})X_{3},\varphi X_{4})=0$ & $\eta
-$Einstein & $(\kappa)g+(-\kappa)\eta\otimes\eta$\\\hline
$g(\mathcal{W}_{6}(\varphi X_{1},X_{2})X_{3},\varphi X_{4})=0$ & Einstein &
$(2n\kappa)g$\\\hline
$g(\mathcal{W}_{8}(\varphi X_{1},X_{2})X_{3},\varphi X_{4})=0$ & Einstein &
$(2n\kappa)g$\\\hline
$g(\mathcal{W}_{9}(\varphi X_{1},X_{2})X_{3},\varphi X_{4})=0$ & Einstein &
$(\frac{2n(\kappa+2n-2)}{2n+1})g$\\\hline
\end{tabular}
\caption{\label{quasiT-flat} Classification of quasi-$ \mathcal{T} $-flat \NKCMM\ with special values of $ a_i $}
\end{table}
\begin{proof}
	Let $ M $ be a quasi-$ \mathcal{T} $-flat \NKCMM. Then,  we have 
\begin{align}
0 &  =a_{0}R(\varphi
X_{1},X_{2},X_{3},\varphi X_{4}) \nonumber\\
& +a_{1}S(X_{2},X_{3})g(\varphi X_{1},\varphi
X_{4})+a_{2}S(\varphi X_{1},X_{3})g(X_{2},\varphi X_{4})\label{Tquasi}\\
&  +a_{3}S(\varphi X_{1},X_{2})g(X_{3},\varphi X_{4})+a_{4}g(X_{2}%
,X_{3})S(\varphi X_{1},\varphi X_{4}) \nonumber \\
& +a_{5}g(\varphi X_{1},X_{3}%
)S(X_{2},\varphi X_{4})  +a_{6}g(\varphi X_{1},X_{2})S(X_{3},\varphi X_{4}) \nonumber\\
& +a_{7}r(g(X_{2}%
,X_{3})g(\varphi X_{1},\varphi X_{4})-g(\varphi X_{1},X_{3})g(X_{2},\varphi
X_{4})).\nonumber
\end{align}
For a $\left\{  e_{1},e_{2},...,e_{2n},\xi\right\}  $ is a local orthonormal
basis of $TM$ and  from (\ref{Tquasi}), we get
\begin{align*}
0&  =a_{0}\overset{2n}{\underset{i=1}{\sum}}R(\varphi
e_{i},X_{2},X_{3},\varphi e_{i}) \\
& +\overset{2n}{\underset{i=1}{\sum}[}%
a_{1}S(X_{2},X_{3})g(\varphi e_{i},\varphi e_{i})+a_{2}S(\varphi e_{i}%
,X_{3})g(X_{2},\varphi e_{i})\\
&  +a_{3}S(\varphi e_{i},X_{2})g(X_{3},\varphi e_{i})+a_{4}g(X_{2}%
,X_{3})S(\varphi e_{i},\varphi e_{i})\\
&  +a_{5}g(\varphi e_{i},X_{3}%
)S(X_{2},\varphi e_{i})+a_{6}g(\varphi e_{i},X_{2})S(X_{3},\varphi e_{i})\\
&  +a_{7}r(g(X_{2}%
,X_{3})g(\varphi e_{i},\varphi e_{i})-g(\varphi e_{i},X_{3})g(X_{2},\varphi
e_{i}))].
\end{align*}
Thus, with using (\ref{Mgei}), (\ref{Mgei2}) and (\ref{Sxe}), we obtain
\begin{align*}
0 &  =a_{0}\left(  S(X_{2},X_{3})-\kappa g(\varphi
X_{2},\varphi X_{3})\right)  +2na_{1}S(X_{2},X_{3})+a_{2}\left(  S(X_{2}, X_{3})-\eta(X_{2})S(X_{3},\xi)\right) \\
 & +a_{3}\left(  S(X_{2},X_{3})-\eta(X_{3})S(X_{2},\xi)\right)  +a_{4}\left(
(r-2n\kappa)g(X_{2},X_{3})\right) \\
&   +a_{5}\left(  S(X_{2},X_{3})-\eta(X_{3}%
)S(X_{2},\xi)\right) +a_{6}\left(  S(X_{2},X_{3})-\eta(X_{2})S(X_{3},\xi)\right)\\
&    +a_{7}%
r(2ng(X_{2},X_{3})-g(\varphi X_{2},\varphi X_{3})].
\end{align*}
Finally,  we get
\begin{align*}
\left[  a_{0}+2na_{1}+a_{2}+a_{3}+a_{5}+a_{6}\right]  S(X_{2},X_{3})  &
=\left[  a_{0}\kappa+a_{4}(2n\kappa-r)+a_{7}r(1-2n)\right]  g(X_{2},X_{3})+\\
&  \left[  -a_{0}\kappa+2n\kappa(a_{2}+a_{3}+a_{5}+a_{6})-a_{7}r\right]  \eta
(X_{2})\eta(X_{3})
\end{align*}
which completes the proof. 

\end{proof}
\begin{corollary}
	If a \NKCMM  \ satisfies $ g(\mathcal{W}_{0}(\varphi X_{1},X_{2})X_{3},\varphi X_{4})=0 $, then it is Ricci flat. 
\end{corollary}

\begin{theorem}
	A \NKCMM\ is $\varphi-\mathcal{T}-$flat if $ C_2=a_{0}+2na_{1}+a_{2}+a_{3}+a_{5}+a_{6}\neq0 $. Then we have 
	\[
	S=\frac{A_2}{C_2}g+(2n\kappa-\frac{A_2}{C_2})\eta\otimes\eta,
	\]
	where%
	\[
	A_2=a_{0}\kappa+a_{4}(2n\kappa-r)+a_{7}r(1-2n)
	\]
	Therefore $M~$is an $\eta-$Einstein manifold. Consequently, we have the classification in Table \ref{fi-T-flat}.
	
\end{theorem}

\begin{table}\small
	
	\begin{tabular}{|l|l|l|}
		\hline
		$N(\kappa)-$contact manifold & $\!\begin{aligned}[t] \eta-\text{Einstein}/ \\ \text{Einstein} \end{aligned}$ & $S=$ \\ \hline
		$g(\mathcal{C}_{\star }(\varphi X_{1},\varphi X_{2})\varphi X_{3},\varphi
		X_{4})=0$ & $\eta -$Einstein 
		& \tiny$\!\begin{aligned}[t] (\frac{(2 n-1) (\kappa+2 n-2) (4 a_{1} n+a_{0})+(2n+1)(\kappa-4a_1n\kappa(n-1))}{(2n+1)2 a_{1} (n-1)+1})g+ \\ 
		
		(2n\kappa-(\frac{\frac{(2 n-1) (\kappa+2 n-2) (4 a_{1} n+a_{0})}{2 n+1}+2 \kappa a_{1} n-2 a_{1} n (\kappa+2 n-2)+\kappa}{2 a_{1} (n-1)+1})) \eta \otimes \eta  \end{aligned} $ \\ \hline
	$g(\mathcal{C}(\varphi X_{1},\varphi X_{2})\varphi X_{3},\varphi X_{4})=0$ & 
	$\eta -$Einstein & $\!\begin{aligned}[t] 
		(2n-2)g+ ( 2n\kappa-2n+2) \eta \otimes \eta \end{aligned} $ \\ \hline
	$g(\mathcal{L}(\varphi X_{1},\varphi X_{2})\varphi X_{3},\varphi X_{4})=0$ & 
	$\eta -$Einstein & $\!\begin{aligned}[t] \left( (2n-1)\kappa+4n(n-1)\right) g+\\ \left( \kappa-4n(n-1)\right)
		\eta \otimes \eta  \end{aligned} $ \\ \hline
	$g(\mathcal{V}(\varphi X_{1},\varphi X_{2})\varphi X_{3},\varphi X_{4})=0$ & 
	$\eta -$Einstein & $\!\begin{aligned}[t] (\frac{2n(2n-3+2\kappa)+2}{2n+1})g+\\ (2n\kappa-(\frac{2n(2n-3+2\kappa)+2}{2n+1})) \eta \otimes \eta  \end{aligned}$ \\ \hline
	$g(\mathcal{P}_{\star }(\varphi X_{1},\varphi X_{2})\varphi X_{3},\varphi
	X_{4})=0$ & $\eta -$Einstein & $ \!\begin{aligned}[t] (\frac{\frac{(2 n-1) (\kappa+2 n-2) (2 a_{1} n+a_{0})}{2 n+1}+\kappa a_{0}}{a_{1} (2 n-1)+a_{0}})g+\\ ( 2n\kappa-(\frac{\frac{(2 n-1) (\kappa+2 n-2) (2 a_{1} n+a_{0})}{2 n+1}+\kappa a_{0}}{a_{1} (2 n-1)+a_{0}}) \eta \otimes \eta \end{aligned}$  \\ \hline
	$g(\mathcal{P}(\varphi X_{1},\varphi X_{2})\varphi X_{3},\varphi X_{4})=0$ & 
	Einstein & $2n\kappa g$ \\ \hline
	$g(\mathcal{M}(\varphi X_{1},\varphi X_{2})\varphi X_{3},\varphi X_{4})=0$ & 
	$\eta -$Einstein & $\left( \frac{2n(\kappa+n-1)}{n+1}\right) g+\left( \frac{%
		2n(n\kappa-n+1)}{n+1}\right) \eta \otimes \eta $ \\ \hline
	$g(\mathcal{W}_{0}(\varphi X_{1},\varphi X_{2})\varphi X_{3},\varphi X_{4})=0
	$ & Einstein & $2n\kappa g$ \\ \hline
	$g(\mathcal{W}_{0}^{_{\star }}(\varphi X_{1},\varphi X_{2})\varphi
	X_{3},\varphi X_{4})=0$ & $\eta -$Einstein & $\left( \frac{2n\kappa}{4n-1}\right)
	g+\left( \frac{8n^{2}\kappa-4n\kappa}{4n-1}\right) \eta \otimes \eta $ \\ \hline
	$g(\mathcal{W}_{1}(\varphi X_{1},\varphi X_{2})\varphi X_{3},\varphi X_{4})=0
	$ & Einstein & $(\frac{2n\kappa}{4n-1})g+ \left( \frac{8n^{2}\kappa-4n\kappa}{4n-1}\right) \eta \otimes \eta$ \\ \hline
	$g(\mathcal{W}_{1}^{_{\star }}(\varphi X_{1},\varphi X_{2})\varphi
	X_{3},\varphi X_{4})=0$ & Einstein & $2n\kappa g$ \\ \hline
	$g(\mathcal{W}_{2}(\varphi X_{1},\varphi X_{2})\varphi X_{3},\varphi X_{4})=0
	$ & $\eta -$Einstein & $\left( \frac{2n(2n-2+\kappa)}{2n+1}\right) g+\left( \frac{%
		4n(n\kappa-n+1)}{2n+1}\right) \eta \otimes \eta $ \\ \hline
	$g(\mathcal{W}_{3}(\varphi X_{1},\varphi X_{2})\varphi X_{3},\varphi X_{4})=0
	$ & $\eta -$Einstein & $\left( \frac{2n(\kappa-2n+2)}{2n-1}\right) g+\left( \frac{%
		2n(2n\kappa-\kappa+2n-3)}{2n-1}\right) \eta \otimes \eta $ \\ \hline
	$g(\mathcal{W}_{4}(\varphi X_{1},\varphi X_{2})\varphi X_{3},\varphi X_{4})=0
	$ & $\eta -$Einstein & $\kappa g+\kappa(2n-1)\eta \otimes \eta $ \\ \hline
	$g(\mathcal{W}_{5}(\varphi X_{1},\varphi X_{2})\varphi X_{3},\varphi X_{4})=0
	$ & $\eta -$Einstein & $\kappa g+\kappa(2n-1)\eta \otimes \eta $ \\ \hline
	$g(\mathcal{W}_{6}(\varphi X_{1},\varphi X_{2})\varphi X_{3},\varphi X_{4})=0
	$ & Einstein & $2n\kappa g$ \\ \hline
	$g(\mathcal{W}_{8}(\varphi X_{1},\varphi X_{2})\varphi X_{3},\varphi X_{4})=0
	$ & Einstein & $2n\kappa g$ \\ \hline
	$g(\mathcal{W}_{9}(\varphi X_{1},\varphi X_{2})\varphi X_{3},\varphi X_{4})=0
	$ & $\eta -$Einstein & $\left( \frac{2n(2n-2+\kappa)}{2n+1}\right) g+\left( \frac{%
		4n(n\kappa-n+1)}{2n+1}\right) \eta \otimes \eta $ \\ \hline
	\end{tabular}
\caption{\label{fi-T-flat} Classification of $\varphi-\mathcal{T-}$flat \NKCMM\ with special values of $ a_i $}
	\end{table}

\begin{proof}
	Let $ M $ be a  $\varphi-\mathcal{T}-$flat \NKCMM. Then, we have 
	\begin{align*}
0 &
	=a_{0}R(\varphi X_{1},\varphi X_{2},\varphi X_{3},\varphi X_{4})+a_{1}%
	S(\varphi X_{2},\varphi X_{3})g(\varphi X_{1},\varphi X_{4})\\
	& +a_{2}S(\varphi
	X_{1},\varphi X_{3})g(\varphi X_{2},\varphi X_{4})  +a_{3}S(\varphi X_{1},\varphi X_{2})g(\varphi X_{3},\varphi X_{4})\\
	& +a_{4}g(\varphi X_{2},\varphi X_{3})S(\varphi X_{1},\varphi X_{4}%
	)+a_{5}g(\varphi X_{1},\varphi X_{3})S(\varphi X_{2},\varphi X_{4})\\
	&  +a_{6}g(\varphi X_{1},\varphi X_{2})S(\varphi X_{3},\varphi X_{4})\\
	& +a_{7}r(g(\varphi X_{2},\varphi X_{3})g(\varphi X_{1},\varphi X_{4}%
	)-g(\varphi X_{1},\varphi X_{3})g(\varphi X_{2},\varphi X_{4})).
	\end{align*}
	for all $  X_{1},X_{2},{3},X_{4} \in \Gamma(TM) $. For $\left\{  e_{1},e_{2},...,e_{2n},\xi\right\}  $ local orthonormal
	basis of $TM$, from (\ref{Tquasi}) we get
	\begin{align*}
	0 &  =a_{0}\overset{2n}{\underset{i=1}{\sum
	}}R(\varphi e_{i},\varphi X_{2},\varphi X_{3},\varphi e_{i})\\
	& +\overset
	{2n}{\underset{i=1}{\sum}[}a_{1}S(\varphi X_{2},\varphi X_{3})g(\varphi
	e_{i},\varphi e_{i})+a_{2}S(\varphi e_{i},\varphi X_{3})g(\varphi
	X_{2},\varphi e_{i})\\
	&  +a_{3}S(\varphi e_{i},\varphi X_{2})g(X_{3},\varphi e_{i})+a_{4}g(\varphi
	X_{2},\varphi X_{3})S(\varphi e_{i},\varphi e_{i})\\
	& +a_{5}g(\varphi
	e_{i},\varphi X_{3})S(\varphi X_{2},\varphi e_{i})
	+a_{6}g(\varphi e_{i},\varphi X_{2})S(\varphi X_{3},\varphi e_{i})\\
	& +a_{7}r(g(\varphi X_{2},\varphi X_{3})g(\varphi e_{i},\varphi e_{i})-g(\varphi e_{i},\varphi X_{3})g(\varphi X_{2},\varphi e_{i}))].
	\end{align*}
	Thus, from (\ref{Mgei}), (\ref{Mgei2}) and (\ref{Sxe}), we have
	\begin{align*}
	0&  =a_{0}\left(  S(\varphi X_{2},\varphi
	X_{3})-\kappa g(\varphi X_{2},\varphi X_{3})\right)  +2na_{1}S(\varphi X_{2},\varphi
	X_{3})+a_{2}S(\varphi X_{2},\varphi X_{3})\\
	&  +a_{3}S(\varphi X_{2},\varphi X_{3})+a_{4}\left(  (r-2n\kappa)g(\varphi
	X_{2},\varphi X_{3})\right)+a_{5}S(\varphi X_{2},\varphi X_{3})  +a_{6}S(\varphi X_{2},\varphi X_{3})\\
	&+a_{7}r(2n-1)g(\varphi X_{2},\varphi
	X_{3}).
	\end{align*}
	and finally, we get 
	\[
	\left[  a_{0}+2na_{1}+a_{2}+a_{3}+a_{5}+a_{6}\right]  S(\varphi X_{2},\varphi
	X_{3})=\left[  a_{0}\kappa+a_{4}(2n\kappa-r)+a_{7}r(1-2n)\right]  g(\varphi
	X_{2},\varphi X_{3})
	\]
	By takig $X_{2}=\varphi X_{2},~X_{3}=\varphi X_{3}$ in last equation we complete the proof. 
\end{proof}

\begin{theorem}
	 A \NKCMM  \ is $\xi-\mathcal{T}-$ flat if $  a_{4}\neq0 $.  Then we have 
\[
S=\left(  \frac{-A_3}{a_{4}}\right)  g+\left(  \frac{-B_3}{a_{4}}\right)
\eta\otimes\eta,
\]
where%
\[
A_3=a_{0}\kappa+2n\kappa a_{1}+2n(2n-2+\kappa)a_{7}%
\]
and%
\[
B_3=-a_{0}\kappa+2n\kappa(a_{2}+a_{3}+a_{5}+a_{6})-2n(2n-2+\kappa)a_{7}%
\]
Therefore $M~$is an $\eta-$Einstein manifold. Consequently, we have the classification in Table \ref{xi-T-flat}.

\begin{table}
	\begin{tabular}
[c]{|l|l|l|}\hline
$N(\kappa)-$contact manifold & $\!\begin{aligned}[t] \eta-\text{Einstein}/\\ \text{Einstein} \end{aligned}$ & $S=$\\\hline
 $\xi-$quasi-conformally flat &
$\eta-$Einstein&  \footnotesize $\!\begin{aligned}[t] \left(  -\frac{-\frac{2 n (\kappa+2 n-2) \left(-2 a_{1}-\frac{a_{0}}{2 n}\right)}{2 n+1}-4 \kappa a_{1} n-\kappa a_{0}}{a_{1}} \right)  g+\\ \left( \frac{2 ((\kappa-1) n+1) (4 a_{1} n+a_{0})}{2 a_{1} n+a_{1}} \right)  \eta\otimes\eta \end{aligned}
$\\\hline
$\xi-$conformally flat & $\eta
-$Einstein & $\!\begin{aligned}[t] \left(2n-2\right)    g+\\ 2\left(n\kappa-n+1\right)    \eta\otimes\eta \end{aligned}
$\\\hline
 $\xi-$conharmonically flat & $\eta
-$Einstein & $  -\kappa  g+\left(  2n\kappa+\kappa\right)  \eta\otimes\eta
$\\\hline
$\xi-\mathcal{M-}$projectively flat &
Einstein & $  2n\kappa  g$\\\hline
 $\xi-\mathcal{W}_{2}$ flat &
Einstein & $  2n\kappa  g$\\\hline
 $\xi-\mathcal{W}_{3}$ flat &
Einstein & $  -2n\kappa  g+  4n\kappa \eta\otimes\eta $\\\hline
 $\xi-\mathcal{W}_{7}$ flat &
$\eta-$Einstein & $2n\kappa\eta\otimes\eta$\\\hline
 $\xi-\mathcal{W}_{9}$ flat &
Einstein & $  2n\kappa  g$\\\hline
\end{tabular}

\caption{\label{xi-T-flat} Classification of $\varphi-\mathcal{T-}$flat \NKCMM\ with special values of $ a_i $}
\end{table}

\end{theorem}

\begin{proof}
Let take $X_{2}=X_{3}=\xi$ in (\ref{gT-tensor}),  we have
\begin{align*}
\mathcal{T}(X_{1},\xi,\xi,X_{4})  &  =a_{0}\kappa (g(X_{1},X_{4})-\eta(X_{1}%
)\eta(X_{4}))+a_{1}2n\kappa g(X_{1},X_{4})+a_{2}2n\kappa \eta(X_{1})\eta(X_{4})\\
&  +a_{3}2n\kappa \eta(X_{1})\eta(X_{4})+a_{4}S(X_{1},X_{4})+a_{5}2n\kappa \eta(X_{1}%
)\eta(X_{4})\\
&  +a_{6}2n\kappa \eta(X_{1})\eta(X_{4})+a_{7}r(g(X_{1},X_{4})-\eta(X_{1})\eta
(X_{4})).
\end{align*}
If $ M $ is   $\xi-\mathcal{T}-$ flat, we obtain
\begin{align*}
-a_{4}S(X_{1},X_{4})  &  =\left[  a_{0}\kappa+a_{1}2n\kappa+2n(2n-2+\kappa)a_{7}\right]
g(X_{1},X_{4})\\
&  +\left[  -a_{0}\kappa+(a_{2}+a_{3}+a_{5}+a_{6})2n\kappa-2n(2n-2+\kappa)a_{7}\right]
\eta(X_{1})\eta(X_{4}). 
\end{align*}
Thus we complete the proof.
\end{proof}

\section{$ \mathcal{T} $-symmetry conditions on $ N(\kappa) $-contact metric manifolds }
A Riemann manifold is called locally symmetric if $ R.R=0 $, where $ . $ denotes a derivative operations on Riemannian curvature tensor $ R $. Similarly one can define symmetry conditions for special curvature tensors such as conformal, concircular, conharmonic, etc. In this section, we consider \T\ for \NK\ with some symmetry conditions. \par 
For two $ (1,3) -$type tensors $ \mathcal{P}_1,\mathcal{P}_2 $ we have 
\begin{eqnarray}\label{generaltensorproduct}
(\mathcal{P}_1(X_1,X_2).\mathcal{P}_2)(X_3,X_4)X_5&=&\mathcal{P}_1(X_1,X_2)\mathcal{P}_2(X_3,X_4)X_5-\mathcal{P}_2(\mathcal{P}_1(X_1,X_2)X_3,X_4)X_5\\&&-\mathcal{P}_2(X_3,\mathcal{P}_1(X_1,X_2)X_4)X_5-\mathcal{P}_2(X_3,X_4)\mathcal{P}_1(X_1,X_2)X_5\notag 
\end{eqnarray}
and for $ (0,2 )-$type tensor $ \mathcal{\rho} $ we have 
\begin{equation}\label{generalT.S}
(\mathcal{P}_1(X_1,X_2).\mathcal{\rho})(X_3,X_4)=\mathcal{\rho}(\mathcal{P}_1(X_1,X_2)X_3,X_4)+\mathcal{\rho}(X_3,\mathcal{P}_1(X_1,X_2)X_4).
\end{equation}
\begin{theorem}
Let $M$ be a $(2n+1)$-dimensional $N(\kappa)$ contact manifold satisfying $ 
\mathcal{T}(\xi,X_{1}).R=0\label{T.Rsart} $for all vector fields $ X_1 $ on $ M $. 
Then $ S=\frac{A}{C}g+\frac{B}{C}\eta\otimes\eta $  if $ C=a_{1}+a_{5}\neq0 $, where $ A=-2n\kappa (a_{2}+a_{4}) $, $ B=2n\kappa (a_{1+}a_{2}+a_{4}+a_{5}) $.
Therefore $M~$is an $\eta-$Einstein manifold. 
Consequently, we have the classification in Table \ref{T.R=0}.
\end{theorem}
\begin{table}
		\begin{center}
			\begin{tabular}
[c]{|l|l|l|}\hline
$N(\kappa)-$contact manifold & $\eta-$Einstein/ Einstein & $S=$\\\hline
$\mathcal{P}_{\star}(\xi,X_1).R=0$ &
Einstein & $2n\kappa g$\\\hline
$\mathcal{W}_{1}(\xi,X_1).R=0$ & Einstein &
$2n\kappa g$\\\hline
$\mathcal{W}_{1}^{_{\star}}(\xi,X_1).R=0$ &
Einstein & $2n\kappa g$\\\hline
$\mathcal{W}_{2}(\xi,X_1).R=0$ & Einstein &
$2n\kappa g$\\\hline
$\mathcal{W}_{4}(\xi,X_1).R=0$ & $\eta
-$Einstein & $2n\kappa \eta\otimes\eta$\\\hline
$\mathcal{W}_{5}(\xi,X_1).R=0$ & Einstein &
$2n\kappa g$\\\hline
$\mathcal{W}_{6}(\xi,X_1).R=0$ & $\eta
-$Einstein & $2n\kappa \eta\otimes\eta$\\\hline
$\mathcal{W}_{7}(\xi,X_1).R=0$ & Einstein &
$2n\kappa g$\\\hline
$\mathcal{W}_{8}(\xi,X_1).R=0$ & $\eta
-$Einstein & $2n\kappa \eta\otimes\eta$\\\hline
\end{tabular}
\caption{\label{T.R=0} Classification of \NKCMM\ satisfying $ \mathcal{T}(\xi,X_{1}).R=0 $ with special values of $ a_i $}
\end{center}
\end{table}
\begin{proof}
	From (\ref{generaltensorproduct}) we get 
\begin{align}
\mathcal{T}(\xi,X_{1}).R  & =\mathcal{T}(\xi,X_{1}).R(X_{2},X_{3}%
)X_{4}-R(\mathcal{T}(\xi,X_{1})X_{2},X_{3})X_{4}\label{T.R}\\
& -R(X_{2},\mathcal{T}(\xi,X_{1})X_{3})X_{4}-R(X_{1},X_{2})\mathcal{T}%
(\xi,X_{1})X_{4}\nonumber.
\end{align}
By taking the inner product
with $\xi$ and then putting $ X_2=X_4=\xi $ in (\ref{T.R}) from (\ref{T.Rsart}), we have%
\[
2\kappa \eta(X_{3})\eta(\mathcal{T}(\xi,X_{1}))\xi-\kappa \eta(\mathcal{T}(\xi,X_{1}%
)X_{3})-\kappa g(X_{3},\mathcal{T}(\xi,X_{1}))\xi)=0.
\]
Thus, we obtain
\begin{equation*}
\left[  a_{1}+a_{5}\right]  S(X_{1},X_{3})=2n\kappa \left[  a_{2}+a_{4}\right]
g(X_{2},X_{3})+2n\kappa \left[  a_{1}+a_{2}+a_{4}+a_{5}\right]  \eta(X_{2}%
)\eta(X_{3})
\end{equation*}
which gives the proof. 
\end{proof}

\begin{theorem}
 Let $M$ be a $(2n+1)$-dimensional $N(k)$ contact manifold satisfying $ \mathcal{T}(\xi,X_{1}).S=0 $ and $ C_5=-a_{1}-a_{5}\neq 0 $. If $ \xi $ is Killing vector field then we have 
\begin{align*}
S(X_{1},X_{2}) & =\frac{A_5}{C_5}g(X_{1},X_{2})+\frac{B_5}{C_5}\eta(X_{1})\eta(X_{2})
\end{align*}
where%
\[
A_{5}=(2n\kappa-2n+2)(a_{0}\kappa+a_{7}r)+4n\kappa(n-1)a_{2}+4n^{2}\kappa^{2}a_{4}%
\]
and%
\[
B_{5}=(2n-2n\kappa-2)(a_{0}\kappa+a_{7}r)+4n^{2}\kappa^{2}(a_{1}+2a_{2}+2a_{3}+a_{4}%
+2a_{5}+2a_{6})-4n\kappa(n-1)(a_{2}+a_{5}). 
\]
Consequently, we have the classification in Table \ref{T.S=0}.
\begin{table}\small
\begin{tabular}
	[c]{|l|l|l|}\hline
	$N(k)-$contact manifold & $ \begin{aligned}
	\eta-\text{Einstein}\\\text{Einstein}
	\end{aligned} $& $S=$\\\hline
	$\mathcal{P}_{\star}(\xi,X_{1}).S=0$ & $\eta-$Einstein & \scriptsize $\!\begin{aligned}[t] (\frac{2 ((\kappa-1) n+1) (\kappa (a_{0}-1)+2 (n-1) (2 n+a_{0}))}{2 a_{1} n+a_{1}}+4 \kappa (n-1) n)g \\ +(\frac{2 ((\kappa-1) n+1) (2 \kappa a_{1} n (2 n+1)-\kappa a_{0}+\kappa-2 (n-1) (2 n+a_{0}))}{2 a_{1} n+a_{1}})\eta\otimes\eta \end{aligned} $\\\hline
	$\mathcal{P}(\xi,X_{1}).S=0$ & Einstein & $(4n^{2}\kappa^{2})g$\\\hline
	$\mathcal{W}_{1}(\xi,X_{1}).S=0$ & $\eta-$Einstein & $(4n\kappa(-n\kappa
	+2n-2))g+(8n\kappa(-n+1+n\kappa))\eta\otimes\eta$\\\hline
	$\mathcal{W}_{1}^{_{\star}}(\xi,X_{1}).S=0$ & Einstein & $(4n^{2}\kappa^{2})g$\\\hline
	$\mathcal{W}_{2}(\xi,X_{1}).S=0$ & $\eta-$Einstein & $4\kappa n(n-1)g$\\\hline
	$\mathcal{W}_{4}(\xi,X_{1}).S=0$ & $\eta-$Einstein & $4\kappa n(n-n\kappa
	-1)g+(4n^{2}\kappa^{2})\eta\otimes\eta$\\\hline
	$\mathcal{W}_{5}(\xi,X_{1}).S=0$ & $\eta-$Einstein & $-4\kappa n(2+n(\kappa-2))
	 g-4\kappa n(n-n\kappa-1)\eta\otimes\eta$\\\hline
	$\mathcal{W}_{6}(\xi,X_{1}).S=0$ & $\eta-$Einstein & $-4\kappa n(n-n\kappa-1)$
	$g+4\kappa n(n-1)\eta\otimes\eta$\\\hline
	$\mathcal{W}_{7}(\xi,X_{1}).S=0$ & $\eta-$Einstein & $4\kappa n(1+n(2\kappa
	-1))g+4\kappa n(n-n\kappa-1)\eta\otimes\eta$\\\hline
	$\mathcal{W}_{8}(\xi,X_{1}).S=0$ & $\eta-$Einstein & $4\kappa n(1+n(\kappa-1)g +4\kappa n(n-1)\eta\otimes\eta$\\\hline
\end{tabular}
\caption{\label{T.S=0} Classification of \NKCMM\ satisfying $ \mathcal{T}(\xi,X_{1}).S=0 $ with special values of $ a_i $}
\end{table}
\end{theorem}
\begin{proof}
	From (\ref{generalT.S}), we get
	\begin{equation}
	S(\mathcal{T}(\xi,X_{1})X_{2},X_{3})+S(X_{2},\mathcal{T}(\xi,X_{1}%
	)X_{3})=0.\label{T.S}%
	\end{equation}
	Putting $X_{3}=\xi$ in (\ref{T.S}), we have%
	\begin{align}
	0  & =2n\kappa [ (a_{0}\kappa+a_{7}r)(g(X_{1},X_{2})-\eta(X_{2})\eta(X_{1}))+a_{1}S(X_{1},X_{2})+\nonumber\\ & 2n\kappa(a_{2}+a_{3}+a_{5}+a_{6})\eta
	(X_{2})\eta(X_{1})+a_{4}2n\kappa g(X_{1},X_{2}) ]  +\label{T.S.1}\\
	& 2(n-1)[  (a_{0}\kappa+a_{7}r)(\eta(X_{2})\eta(X_{1})-g(X_{1}%
	,X_{2}))+a_{2}2n\kappa g(X_{1},X_{2})+ \nonumber\\ & a_{5}S(X_{1},X_{2})+2n\kappa(a_{1}+a_{3}+a_{4}+a_{6})\eta(X_{2})\eta(X_{1})]  +\nonumber\\
	& 2(n-1)\left[  -(a_{0}\kappa+a_{7}r)g(X_{1},hX_{2})+a_{2}2n\kappa
	g(X_{1},hX_{2})+a_{5}S(X_{1},hX_{2})\right]  +\nonumber\\
	& 2(n\kappa-(n-1))\eta(X_{2})\left[  2n\kappa\eta(X_{1})(a_{1}+a_{2}%
	+a_{3}+a_{4}+a_{5}+a_{6}) \right]  \nonumber
	\end{align}
	If $h=0$ in (\ref{T.S.1}),i.e $ \xi $ is Killing vector field, The proof is completed. 
\end{proof}

\end{document}